\documentclass[12pt]{article}

\usepackage{amssymb}
\usepackage{amsmath}
\usepackage{booktabs}
\usepackage{multicol}

\usepackage{amsthm}
\newtheorem{thm}{Theorem}

\newtheorem{lem}{Lemma}

\newcommand{\B}{{\mathcal B}}
\newcommand{\F}{{\mathbb F}}
\newcommand{\LS}{{\mathcal L}}
\newcommand{\qbinom}[3]{\genfrac{[}{]}{0pt}{}{#1}{#2}_{#3}}
\newcommand{\subspaces}[2]{\genfrac{[}{]}{0pt}{}{#1}{#2}}
\renewcommand{\mod}{\Modulo}
\DeclareMathOperator\GL{GL}

\DeclareMathOperator\PG{PG}
\DeclareMathOperator\Modulo{mod}

\begin{document}


\title{Large Sets of $t$-Designs over Finite Fields}
\author{\sc Michael Braun\\
University of Applied Sciences Darmstadt, Germany\\[1ex]
\sc Axel Kohnert\\
University of Bayreuth, Germany\\[1ex]
\sc Patric \"Osterg\aa rd\\
Aalto University, Finland\\[1ex]
\sc Alfred Wassermann\\
University of Bayreuth, Germany}
\maketitle

\begin{abstract}
A $t\text{-}(n,k,\lambda;q)$-design is a set of $k$-subspaces, called blocks, of an $n$-dimensional vector space $V$ over the finite field with $q$ elements such that each $t$-subspace is contained in exactly $\lambda$ blocks. A partition of the complete set of $k$-subspaces of $V$ into disjoint $t\text{-}(n,k,\lambda;q)$ designs is called a large set of $t$-designs over finite fields. In this paper we give the first nontrivial construction of such a large set with $t\ge2$.
\end{abstract}

\newpage


\section{Introduction}
\label{sec:intro}

A simple $t$-\emph{design over a finite field} or, more precisely, a $t\text{-}(n,k,\lambda;q)$ \emph{design} is a set 
$\B$ of $k$-subspaces of an $n$-dimensional vector space $V$ over the finite field $\F_q$ such that each $t$-subspace of $V$ is contained in exactly $\lambda$ members of $\B$. 

The study of combinatorial $t$-designs and Steiner systems on (finite) sets goes back to the 19th century 
and has a rich literature \cite{Handbook2007}.
Cameron~\cite{Cameron1,Cameron2} and Delsarte~\cite{Delsarte} extended 
the notions of $t$-designs and Steiner systems from sets to vector spaces over finite fields in the early 1970s. 
Recently, designs over finite fields gained a lot of interest because of applications for
error-correction in networks \cite{KK08a}.
    
In 1987, Thomas \cite{Tho87} constructed the first nontrivial simple $t$-designs over finite fields
for $t=2$.
Since then, more designs over finite fields have been constructed, see
\cite{BKL05,BEO+13,Ito98,MMY95,Suz90,Suz92}. 
Specifically, in \cite{BKL05} the first nontrivial $t$-design over finite fields with $t=3$  has been found
and in \cite{BEO+13} $2\text{-}(13,3,1;2)$ designs have been constructed. The latter ones are the first nontrivial $t$-designs 
over finite fields with $\lambda=1$ and $t=2$. Designs with $\lambda=1$ are called $q$-Steiner systems.

An $LS_q[N](t,k,n)$ \emph{large set} $\LS$ is a set of $N$ disjoint $t\text{-}(n,k,\lambda;q)$ designs 
such that their union forms the complete set of all $k$-subspaces of $V=\F_q^n$. 
Large sets of designs over finite fields have been studied for the first time by Ray-Chaudhuri and Schram \cite{RS94}.
There, the authors used non-simple designs.
In this paper we investigate the existence of large sets of \emph{simple} $t$-designs over finite fields.

In the case of designs on sets, large sets are intensively studied objects \cite[Section II.4.4]{KL07}.
A celebrated result by Teirlinck \cite{Teir89} is that large sets of designs on sets exist for all $t>0$ and $N>0$.



Large sets of certain $t\text{-}(n,k,\lambda;q)$ designs have been intensively studied in the framework of projective geometry. 
In geometry, $1\text{-}(n,k,1;q)$ designs are known as $(k-1)$-\emph{spreads} in $\PG(n-1,q)$.
A large set of $1\text{-}(n,k,1;q)$ designs is called $(k-1)$-\emph{parallelism} of the projective geometry $\PG(n-1,q)$. 
A \emph{parallelism} is a $1$-parallelism, i.\,e. $k=2$.

Since $1\text{-}(n,k,1;q)$ designs exist if and only if $k$ divides $n$, a necessary condition for the existence 
of a parallelism in $\PG(n-1,q)$ is that $n$ must be even. Beutelspacher \cite{Beu74} proved the existence of a 
parallelism in $\PG(2^i-1,q)$ for all $i\ge 2$. Later, Baker \cite{Bak76} and Wettl \cite{Wet91} gave a 
construction of parallelisms in $\PG(n-1,q)$ for $n$ even. Penttila and Williams \cite{PW98} studied 
$\PG(3,q)$ for $q \equiv 2 \mod 3$ and constructed parallelisms subsuming the results presented in \cite{Lun84}. 

Up to now, no large sets of $t$-designs over finite fields with $t\geq 2$ have been reported.
The main result of this paper is the following one:

\begin{thm}\label{thm:main}
Nontrivial large sets of $t$-designs over finite fields exist for $t\ge 2$.
\end{thm}
The theorem is proved by showing the existence of a large set 
consisting of three disjoint $2\text{-}(8,3,21;2)$ designs.


\section{The Construction of Large Sets}
\label{sec:construction}

Let $\subspaces{V}{k}$ denote the set of $k$-subspaces of $V$.
The expression 
\[
\qbinom{n}{k}{q}=\genfrac{}{}{}{}{(q^n-1)(q^{n-1}-1)\cdots(q^{n-k+1}-1)}{(q^k-1)(q^{k-1}-1)\cdots(q-1)}
\]
is called the $q$-\emph{binomial coefficient}.
The set $\subspaces{V}{k}$ itself is already a design, the so-called \emph{trivial} design,
with parameters $t\text{-}(n,k,\lambda_{\max};q)$,  where 
\[
\lambda_{\max}=\qbinom{n-t}{k-t}{q}.
\]
Hence, an obvious necessary condition for the existence of a $LS_q[N](t,k,n)$ large set is the equality 
$\lambda\cdot N = \lambda_{\max}$. 
Moreover, since the blocks of a $t$-design also form an $i$-design as long as $0\leq i\leq t$, 
we have the necessary conditions
\[
    N \mid \qbinom{n-i}{k-i}{q}\quad\mbox{for }  0\leq i\leq t.
\]

The general linear group $\GL(n,q)$, whose elements are represented by $n\times n$-matrices $\alpha$, 
acts on $\subspaces{V}{k}$ by left multiplication $\alpha K := \{\alpha x\mid x\in K\}$. 
An element $\alpha\in \GL(n,q)$ is called an \emph{automorphism} of a $t$-$(n,k,\lambda;q)$ design $\B$ 
if $\B=\alpha\B := \{\alpha K\mid K\in B\}$. The set of all automorphisms of a design forms a group, called \emph{the automorphism group} of the design. Every subgroup of the automorphism group of a design is denoted as \emph{a group of automorphisms} of the design.


If $G$ is a subgroup of $\GL(n,q)$ the $G$-orbit on a $k$-subspace $K$ is denoted by 
$G(K):=\{\alpha K\mid \alpha \in G\}\subseteq \subspaces{V}{k}$. 
Now, a $t\text{-}(n,k,\lambda;q)$ design $\B$ admits a subgroup $G$ of the general linear $\GL(n,q)$ 
as a group of automorphisms if and only if $\B$ consists of $G$-orbits on $\subspaces{V}{k}$. 
The $G$-incidence matrix $A_{t,k}^G$ is defined to be the matrix whose rows and columns are 
indexed by the $G$-orbits on the set of $t$- and $k$-subspaces of $V$, respectively. 
The entry indexed by the orbit $G(T)$ on $\subspaces{V}{t}$ and the orbit $G(K)$ on $\subspaces{V}{k}$ 
is defined by $|\{K'\in G(K)\mid T\subseteq K'\}|$.

According to Kramer and Mesner \cite{KM76} a simple $t\text{-}(n,k,\lambda;q)$ design admitting $G$ 
as a group of automorphisms exists if and only if there 
is a $0/1$-column vector $\textbf{x}$ satisfying $A_{t,k}^G\textbf{x}=\lambda\textbf{1}$, where $\textbf{1}$ denotes the all-one column vector.
The vector $\textbf{x}$ represents the corresponding selection of $G$-orbits on $\subspaces{V}{k}$.

The following algorithm describes a basic approach to find large sets. 
A version of this algorithm for large sets of designs on sets can be found in \cite{LMW01,LOT+07}.

\paragraph{Algorithm A.} The algorithm computes an $LS_q[N](t,k,n)$ large set $\LS$ consisting of $N$ $t\text{-}(n,k,\lambda;q)$ designs 
admitting $G$ as a group of automorphisms. Either the algorithm terminates with a large set or it ends without any statement about the existence.
\begin{enumerate}
\item[\textbf{A1}.] [\emph{Initialize}.] Set $\textbf{B}$ as the complete set of $G$-orbits on $\subspaces{V}{k}$ and set $\LS:=\emptyset$.
\item[\textbf{A2}.] [\emph{Solve}.] Find a random $t\text{-}(n,k,\lambda;q)$ design $\B$ consisting of orbits of $\textbf{B}$. 
If such a $t$-design exists insert $\B$ into $\LS$ and continue with \textbf{A3}. Otherwise terminate without a large set.
\item[\textbf{A3}.] [\emph{Remove}.] Remove the selected orbits in $\B$ from $\textbf{B}$. If $\textbf{B}=\emptyset$ then terminate 
with a large set $\LS$. Otherwise goto \textbf{A2}.
\end{enumerate}
Algorithm A can be implemented by a slight modification of the Kramer-Mesner approach. 
We just have to add a further row to the Diophantine system of equations the following way:
\[
\left[
\begin{array}{c}
\\
A_{t,k}^G\\
\\[1ex]
\hline
\cdots y_K \cdots
\end{array}\right]\textbf{x}=
\left[
\begin{array}{c}
\lambda\\
\vdots\\
\lambda\\
\hline
0
\end{array}
\right]
\]
The vector $\textbf{y}=[\cdots y_K \cdots]$ is indexed by the $G$-orbits on $\subspaces{V}{k}$ corresponding to the columns of $A_{t,k}^G$. The entry $y_K$ indexed by the $G$-orbit containing $K$ is defined to be one if the orbit has already been covered by a selected $t\text{-}(n,k,\lambda;q)$ design. Otherwise it is zero. In every iteration step the vector $\textbf{y}$ has to be updated.

A second simple approach which might be reasonable if the number of total solutions of the Kramer-Mesner system is small uses an exact cover solver \cite{Knu00}.

\paragraph{Algorithm B.} The algorithm computes an $LS_q[N](t,k,n)$ large set $\LS$ consisting of 
$N$ $t\text{-}(n,k,\lambda;q)$ designs admitting $G$ as a group of automorphisms. 
The algorithm terminates with the existence statement \texttt{true} or \texttt{false}.

\begin{enumerate}
\item[\textbf{B1}.] [\emph{Initialize}.] Find all $0/1$-column vectors $\textbf{x}_1,\ldots,\textbf{x}_s$ solving $A_{t,k}^G\textbf{x}=\lambda\textbf{1}$ and form the matrix
$A=[\textbf{x}_1\mid\cdots\mid \textbf{x}_s]$.

\item[\textbf{B2}.] [\emph{Exact cover}.] Find a $0/1$-vector $\textbf{y}$ solving the system $A\textbf{y}=\textbf{1}$. If such a solution $\textbf{y}$ exists return \texttt{true}. Otherwise return \texttt{false}.
\end{enumerate}


\section{The Existence of $LS_2[3](2,3,8)$}
\label{sec:existence}

In this section we present the construction of the first nontrival large set of designs over finite fields, 
a large set $\LS$ with parameters $LS_2[3](2,3,8)$. The large set consists of
three $2\text{-}(8,3,21;2)$ designs $\LS=\{\B_1, \B_2, \B_3\}$, each admitting $G=\langle\alpha\rangle\le \GL(8,2)$ 
as a group of automorphisms. The group $G$ has order $255$ and is generated by a Singer cycle $\alpha$, represented by the matrix
\[
\alpha=
\begin{bmatrix}
0& 0& 0& 0& 0& 0& 0& 1 \\
1& 0& 0& 0& 0& 0& 0& 0 \\
0& 1& 0& 0& 0& 0& 0& 1 \\
0& 0& 1& 0& 0& 0& 0& 1 \\
0& 0& 0& 1& 0& 0& 0& 1 \\
0& 0& 0& 0& 1& 0& 0& 0 \\
0& 0& 0& 0& 0& 1& 0& 0 \\
0& 0& 0& 0& 0& 0& 1& 0 \\
\end{bmatrix} \, .
\]
The existence of  $2$-$(8,3,21;2)$ designs having $G$ as a group of automorphisms has been shown previously in \cite{BKL05}.
The large set $\LS$ was constructed with Algorithm A. 
Each of the three designs of $\LS$ consists of $127$ orbits of $G$ on the set of $3$-subspaces of $V=\F_2^8$.
The orbit representatives for each of the designs $\B_1$, $\B_2$, and $\B_3$ are depicted in Tables \ref{tab:B1}, \ref{tab:B2}, and \ref{tab:B3}. 
For each representative, the three column vectors 
\[
\begin{bmatrix}
x_0 & y_0 & z_0\\
x_1 & y_1 & z_1\\
\vdots &\vdots&\vdots\\
x_{7} & y_{7} & z_{7}\\
\end{bmatrix}
\]
spanning a $3$-subspace of $V$,
are encoded as a triple of the positive integers
\[
[X,Y,Z]=\left[\sum_{i=0}^{7}x_i2^i,\sum_{i=0}^{7}y_i2^i, \sum_{i=0}^{7}z_i2^i\right].
\]

\begin{table}[!htbp]
\small
{\centering
\caption{Design $\B_1$}\label{tab:B1}}
\begin{multicols}{5}
$[1,112,128]$
$[1,48,128]$
$[2,80,128]$
$[2,96,128]$
$[3,48,128]$
$[3,64,128]$
$[4,32,128]$
$[4,48,128]$
$[4,72,128]$
$[5,72,128]$
$[5,80,128]$
$[6,32,128]$
$[6,72,128]$
$[6,96,128]$
$[7,32,128]$
$[7,48,128]$
$[7,80,128]$
$[8,64,128]$
$[9,64,128]$
$[9,96,128]$
$[10,32,128]$
$[10,64,128]$
$[12,96,128]$
$[13,32,128]$
$[13,64,128]$
$[14,32,128]$
$[14,64,128]$
$[15,16,128]$
$[15,96,128]$
$[16,64,128]$
$[17,64,128]$
$[17,96,128]$
$[18,32,128]$
$[19,8,128]$
$[19,96,128]$
$[20,96,128]$
$[21,24,128]$
$[21,32,128]$
$[21,96,128]$
$[23,8,128]$
$[25,32,128]$
$[33,48,128]$
$[33,64,128]$
$[33,80,128]$
$[34,64,128]$
$[35,4,128]$
$[35,48,128]$
$[36,64,128]$
$[37,40,128]$
$[37,64,128]$
$[38,16,128]$
$[38,112,128]$
$[40,112,128]$
$[41,112,128]$
$[41,64,128]$
$[42,48,128]$
$[43,80,128]$
$[44,16,128]$
$[45,80,128]$
$[47,112,128]$
$[49,40,128]$
$[49,64,128]$
$[50,64,128]$
$[50,84,128]$
$[52,40,128]$
$[55,56,128]$
$[55,64,128]$
$[56,64,128]$
$[57,4,128]$
$[60,64,128]$
$[61,64,128]$
$[62,64,128]$
$[63,64,128]$
$[65,32,128]$
$[66,80,128]$
$[67,48,128]$
$[67,96,128]$
$[69,16,128]$
$[69,80,128]$
$[70,8,128]$
$[70,32,128]$
$[70,96,128]$
$[71,16,128]$
$[71,48,128]$
$[72,80,128]$
$[73,80,128]$
$[74,16,128]$
$[75,80,128]$
$[77,96,128]$
$[78,32,128]$
$[79,32,128]$
$[79,48,128]$
$[79,112,128]$
$[83,8,128]$
$[83,96,128]$
$[84,8,128]$
$[84,96,128]$
$[85,24,128]$
$[85,32,128]$
$[85,96,128]$
$[86,32,128]$
$[87,32,128]$
$[89,96,128]$
$[90,32,128]$
$[92,32,128]$
$[94,32,128]$
$[94,96,128]$
$[98,8,128]$
$[98,16,128]$
$[99,16,128]$
$[99,36,128]$
$[100,8,128]$
$[100,24,128]$
$[100,40,128]$
$[101,16,128]$
$[101,112,128]$
$[102,80,128]$
$[103,40,128]$
$[103,48,128]$
$[103,72,128]$
$[105,112,128]$
$[106,112,128]$
$[110,16,128]$
$[110,48,128]$
$[114,40,128]$
$[114,120,128]$
$[122,4,128]$
\end{multicols}
\end{table}

\begin{table}[!htbp]
\small
{\centering
\caption{Design $\B_2$}\label{tab:B2}}
\begin{multicols}{5}
$[1,24,128]$
$[1,32,128]$
$[1,64,128]$
$[1,92,128]$
$[2,32,128]$
$[2,64,128]$
$[3,16,128]$
$[3,32,128]$
$[3,68,128]$
$[3,80,128]$
$[3,96,128]$
$[5,32,128]$
$[5,64,128]$
$[7,96,128]$
$[8,48,128]$
$[9,16,128]$
$[10,96,128]$
$[12,32,128]$
$[12,64,128]$
$[14,48,128]$
$[14,80,128]$
$[15,32,128]$
$[15,48,128]$
$[15,112,128]$
$[18,64,128]$
$[19,72,128]$
$[20,32,128]$
$[20,64,128]$
$[20,72,128]$
$[20,120,128]$
$[22,32,128]$
$[22,64,128]$
$[23,64,128]$
$[25,64,128]$
$[26,96,128]$
$[27,32,128]$
$[27,64,128]$
$[27,96,128]$
$[28,64,128]$
$[28,96,128]$
$[30,64,128]$
$[31,32,128]$
$[31,96,128]$
$[32,64,128]$
$[34,20,128]$
$[34,80,128]$
$[35,8,128]$
$[35,16,128]$
$[35,36,128]$
$[37,8,128]$
$[37,72,128]$
$[37,112,128]$
$[38,64,128]$
$[38,72,128]$
$[38,80,128]$
$[39,80,128]$
$[39,120,128]$
$[40,16,128]$
$[41,16,128]$
$[41,48,128]$
$[42,64,128]$
$[43,64,128]$
$[44,80,128]$
$[46,16,128]$
$[46,64,128]$
$[47,16,128]$
$[47,80,128]$
$[51,64,128]$
$[52,64,128]$
$[52,72,128]$
$[54,8,128]$
$[55,24,128]$
$[65,24,128]$
$[65,112,128]$
$[66,16,128]$
$[66,48,128]$
$[66,112,128]$
$[67,80,128]$
$[67,112,128]$
$[68,16,128]$
$[68,96,128]$
$[69,24,128]$
$[69,32,128]$
$[69,40,128]$
$[69,96,128]$
$[70,48,128]$
$[71,32,128]$
$[71,72,128]$
$[71,80,128]$
$[71,112,128]$
$[72,32,128]$
$[72,96,128]$
$[73,96,128]$
$[74,32,128]$
$[74,96,128]$
$[76,32,128]$
$[76,48,128]$
$[77,32,128]$
$[78,96,128]$
$[82,32,128]$
$[82,56,128]$
$[83,32,128]$
$[84,32,128]$
$[85,40,128]$
$[86,96,128]$
$[87,8,128]$
$[87,24,128]$
$[87,96,128]$
$[91,32,128]$
$[95,96,128]$
$[98,48,128]$
$[98,80,128]$
$[99,48,128]$
$[99,80,128]$
$[100,16,128]$
$[102,16,128]$
$[102,40,128]$
$[103,16,128]$
$[108,48,128]$
$[109,16,128]$
$[109,48,128]$
$[114,4,128]$
$[114,36,128]$
$[115,40,128]$
$[115,72,128]$
$[117,120,128]$
$[118,120,128]$
\end{multicols}
\end{table}

\begin{table}[!htbp]
\small
{\centering
\caption{Design $\B_3$}\label{tab:B3}}
\begin{multicols}{5}
$[1,16,128]$
$[2,120,128]$
$[2,36,128]$
$[2,52,128]$
$[3,40,128]$
$[4,56,128]$
$[4,64,128]$
$[4,96,128]$
$[5,16,128]$
$[5,96,128]$
$[5,112,128]$
$[6,56,128]$
$[6,64,128]$
$[7,16,128]$
$[7,64,128]$
$[7,112,128]$
$[8,32,128]$
$[8,80,128]$
$[8,96,128]$
$[9,32,128]$
$[9,48,128]$
$[9,80,128]$
$[10,16,128]$
$[13,48,128]$
$[13,96,128]$
$[14,16,128]$
$[14,96,128]$
$[15,64,128]$
$[15,80,128]$
$[17,32,128]$
$[19,32,128]$
$[19,64,128]$
$[19,68,128]$
$[21,64,128]$
$[23,32,128]$
$[23,96,128]$
$[24,32,128]$
$[24,64,128]$
$[26,64,128]$
$[27,4,128]$
$[28,32,128]$
$[30,32,128]$
$[30,96,128]$
$[31,64,128]$
$[33,40,128]$
$[33,72,128]$
$[34,48,128]$
$[34,72,128]$
$[34,100,128]$
$[35,64,128]$
$[35,72,128]$
$[35,80,128]$
$[36,120,128]$
$[37,24,128]$
$[37,80,128]$
$[39,8,128]$
$[39,16,128]$
$[39,48,128]$
$[39,64,128]$
$[40,64,128]$
$[42,80,128]$
$[43,16,128]$
$[44,64,128]$
$[45,64,128]$
$[47,64,128]$
$[48,64,128]$
$[49,8,128]$
$[53,64,128]$
$[54,64,128]$
$[55,72,128]$
$[57,64,128]$
$[58,64,128]$
$[59,64,128]$
$[65,16,128]$
$[65,48,128]$
$[65,80,128]$
$[65,96,128]$
$[66,32,128]$
$[66,96,128]$
$[67,16,128]$
$[67,32,128]$
$[68,32,128]$
$[69,48,128]$
$[69,120,128]$
$[70,16,128]$
$[71,8,128]$
$[71,56,128]$
$[72,48,128]$
$[73,16,128]$
$[73,32,128]$
$[76,16,128]$
$[76,80,128]$
$[76,96,128]$
$[77,112,128]$
$[78,16,128]$
$[79,16,128]$
$[79,96,128]$
$[81,32,128]$
$[81,96,128]$
$[82,8,128]$
$[82,96,128]$
$[83,40,128]$
$[86,24,128]$
$[87,40,128]$
$[87,72,128]$
$[88,32,128]$
$[91,96,128]$
$[92,96,128]$
$[95,32,128]$
$[99,112,128]$
$[100,112,128]$
$[100,120,128]$
$[101,24,128]$
$[101,8,128]$
$[101,48,128]$
$[102,56,128]$
$[104,16,128]$
$[104,112,128]$
$[107,112,128]$
$[107,48,128]$
$[110,80,128]$
$[111,16,128]$
$[114,8,128]$
$[114,52,128]$
$[114,84,128]$
$[118,8,128]$
$[118,24,128]$
\end{multicols}
\end{table}


\section{Further Results}
\label{sec:furtherresults}

Let $K^\perp=\{x\in V\mid \langle x,y\rangle = 0 \, \mbox{ for all } y\in K\}$ denote the orthogonal complement of a subspace $K$ of $V$ with respect to the standard inner product $\langle-,-\rangle$. By Suzuki \cite[Lemma 4.3]{Suz89} we know that every $t\text{-}(n,k,\lambda;q)$ design $\B$ defines a $t\text{-}(n,n-k,\lambda^\perp;q)$ design $\B^\perp:=\{K^\perp\mid K\in \B\}$ with 
\[
\lambda^\perp=\lambda\frac{\qbinom{n-t}{k}{q}}{\qbinom{n-t}{k-t}{q}}.
\]
This design is called the \emph{complementary design}. 
If a $t$-design admits $G$ as a group of automorphisms, which group does the complementary $t$-design admit? 

The orthogonal complement of a subspace corresponds to the set-wise complement of a subset in the classical situation for designs on sets.
There, it is clear that the automorphism group which is a subgroup of the symmetric group remains the same, since set-wise complements commute with permutations: $\overline{\pi K}=\pi \overline{K}$. 

For designs over finite fields the following lemma gives the answer.

\begin{lem}\label{lemma1}
A $t\text{-}(n,k,\lambda;q)$ design $\B$ admits $G\le \GL(n,q)$ as a group of automorphisms if and only if the complementary $t\text{-}(n,n-k,\lambda^\perp;q)$ design $\B^\perp$ admits $H=\{\alpha^T\mid \alpha\in G\}$ as a group of automorphisms.
\end{lem}

\begin{proof} 
We have
\begin{align*}
(\alpha K)^\perp
=&\{x\in V \mid \langle x, \alpha y\rangle = 0\,\forall y\in K\}\\
=&\{x\in V \mid \langle \alpha^T x, y\rangle = 0\,\forall y\in K\}\\
=&\{(\alpha^T)^{-1}x \mid x\in V: \langle x, y\rangle = 0\, \mbox{ for all } y\in K\}\\
=& \{(\alpha^T)^{-1}x \mid x\in K^\perp\}\\
=& (\alpha^T)^{-1}(K^\perp).
\end{align*}
Since the mapping $\alpha\mapsto (\alpha^T)^{-1}$ defines a group isomorphism between $G$ and $H$ the orthogonal complement maps orbits $G(K)$ of $\subspaces{V}{k}$ onto orbits $H(K^\perp)$ of $\subspaces{V}{n-k}$. This completes the proof.
\end{proof}

The orthogonal complement defines a bijection between the set of $k$- and $(n-k)$-subspaces, and hence a partition of $\subspaces{V}{k}$ into $t\text{-}(n,k,\lambda;q)$ designs yields a partition of $\subspaces{V}{n-k}$ into $t\text{-}(n,n-k,\lambda^\perp;q)$ designs. Finally, the existence of an $LS_q[N](t,k,n)$ large set implies the existence of an $LS_q[N](t,n-k,n)$ large set.

Taking the orthogonal complements of each orbit representative of the $2\text{-}(8,3,21;2)$ designs given in the Tables \ref{tab:B1}, \ref{tab:B2}, and \ref{tab:B3}, we obtain representatives of disjoint $2\text{-}(8,5,21;2)$ designs forming an $LS_2[3](2,5,8)$ large set, where each design is 
admitting a Singer cyclic group as a group of automorphisms by Lemma~\ref{lemma1}.


%
\section*{Acknowledgments}
The collaboration of the authors has been partly funded by DFG KO 3154/9-1  and by the European COST project IC1104.
The research of the third author was supported in part by the Academy
 of Finland under Grant No.\ 132122.



\begin{thebibliography}{99}

\bibitem{Bak76}
R.~D.~Baker.
\newblock {Partioning the Planes of $AG_{2m}(2)$ into $2$-Designs}.
\newblock {\em Discrete Mathematics}, 15:205--211, 1976.

\bibitem{Beu74}
A.~Beutelspacher.
\newblock {On Parallelisms of Finite Projective Spaces}.
\newblock {\em Geometriae Dedicata}, 3:35--40, 1974.


\bibitem{BKL05}
M.~Braun, A.~Kerber, and R.~Laue.
\newblock {Systematic Construction of $q$-Analogs of Designs}.
\newblock {\em Designs, Codes and Cryptography}, 34:55--70, 2005.

\bibitem{BEO+13}
M.~Braun, T.~Etzion, P.~J.~R.~\"{O}sterg\aa rd, A.~Vardy, and A.~Wassermann.
\newblock�{Existence of $q$-Analogs of Steiner Systems}.
\newblock submitted, 2013.

\bibitem{Cameron1}
P.~J.~Cameron.
\newblock {Generalisation of Fisher's Inequality to Fields with More than One Element},
\newblock in T.~McDonough and V.~Mavron, Eds., \emph{Combinatorics}, London Mathematical Society Lecture Note Series, 13:9--13, 1974.

\bibitem{Cameron2}
P.~J.~Cameron.
\newblock {Locally Symmetric Designs}. 
\newblock {\em Geometriae Dedicata}, 3:65--76, 1974.
		
\bibitem{Handbook2007}
C.~J.~Colbourn and J.~H.~Dinitz  (eds.).
\newblock {\em Handbook of Combinatorial Designs (2nd ed.)}.
\newblock {CRC Press}, 2007.

\bibitem{Delsarte} 
P.~Delsarte.
\newblock {Association Schemes and $t$-Designs in Regular Semilattices}.
\newblock {\em Journal of Combinatorial Theory, Series A}, 20:230--243, 1976.


\bibitem{Ito98}
T.~Itoh.
\newblock {A New Family of $2$-Designs over $GF(q)$ Admitting $SL_m(q^l)$}.
\newblock {\em Geometriae Dedicata}, 69:261--286, 1998.

\bibitem{KL07}	
G.~B.~Khosrovshahi and R.~Laue.
\newblock {$t$-designs with $t \geq 3$}.
\newblock in C.~J.~Colbourn and J.~H.~Dinitz (eds.), {\em Handbook of Combinatorial Designs (2nd ed.)}, 79--101, CRC Press, 2007

\bibitem{Knu00} 
D.~E.~Knuth.
\newblock Dancing Links.
\newblock in J.~Davies, B.~Roscoe, and J.~Woodcock (eds.), \emph{Millennial Perspectives in Computer Science}, Palgrave Macmillan, Basingstoke, 187--214, 2000.

\bibitem{KK08a} 
R.~Koetter and F.~Kschischang.
\newblock {Coding for Errors and Erasures in Random Network Coding}.
\newblock {\em IEEE Transactions on Information Theory}, 54:3579--3591, 2008.

\bibitem{KM76}
E.~Kramer and D.~Mesner.
\newblock {$t$-Designs on Hypergraphs}.
\newblock {\em Discrete Mathematics}, 15(3):263--296, 1976.

\bibitem{LMW01}
R.~Laue, S.~Magliveras, and A.~Wassermann.
\newblock {New Large Sets of $t$-Designs}.
\newblock {\em Journal of Combinatorial Designs}, 9:40--59, 2001.

\bibitem{LOT+07}
R.~Laue, G.~R. Omidi, B.~Tayfeh-Rezaie, and A.~Wassermann.
\newblock {New Large Sets of $t$-Designs with Prescribed Groups of Automorphisms}.
\newblock {\em Journal of Combinatorial Designs}, 15(3):210--220, 2007.

\bibitem{Lun84}
G.~Lunardon.
\newblock {On Regular Parallelisms in $PG(3,q)$}.
\newblock {\em Discrete Mathematics}, 51:229--235, 1984.

\bibitem{MMY95}
M.~Miyakawa, A.~Munemasa, and S.~Yoshiara.
\newblock {On a Class of Small $2$-Designs over $GF(q)$}.
\newblock {\em Journal of Combinatorial Designs}, 3:61--77, 1995.

\bibitem{PW98}
T.~Penttila and B.~Williams.
\newblock {Regular Packings of $PG(3,q)$}.
\newblock {\em European Journal of Combinatorics}, 19:713--720, 1998.

\bibitem{RS94}
D.~K. Ray-Chaudhuri and E.~J. Schram.
\newblock {A Large Set of Designs on Vector Spaces}.
\newblock {\em Journal of Number Theory}, 47:247--272, 1994.


\bibitem{Suz89}
H.~Suzuki.
\newblock {\em Five Days Introduction to the Theory of Designs}.
\newblock {Lecture Notes, given at Osaka City Univ. in December, 1989.}

\bibitem{Suz90}
H.~Suzuki.
\newblock {$2$-Designs over $GF(2^m)$}.
\newblock {\em Graphs and Combinatorics}, 6:293--296, 1990.

\bibitem{Suz92}
H.~Suzuki.
\newblock {$2$-Designs over $GF(q)$}.
\newblock {\em Graphs and Combinatorics}, 8:381--389, 1992.

\bibitem{Teir89}
L.~Teirlinck.
\newblock {Locally Trivial $t$-Designs and $t$-Designs without Repeated Blocks}.
\newblock {\em Discrete Mathematics}, 77:345--356, 1989. 

\bibitem{Tho87}
S.~Thomas.
\newblock {Designs over Finite Fields}.
\newblock {\em Geometriae Dedicata}, 24:237--242, 1987.

\bibitem{Wet91}
F.~Wettl.
\newblock {On Parallelisms of Odd-Dimensional Finite Projective Spaces}.
\newblock {Proceedings of the second international mathematical miniconference, part II
(Budapest, 1988), Period Polytech. Transportation Engrg}, 19(1-2):111--116, 1991.

\end{thebibliography}
\end{document}